\documentclass{article}

\usepackage{amssymb}
\usepackage{amsmath}

\usepackage{amsthm}

\newtheoremstyle{dotless}
  {6pt}
  {6pt}
  {\upshape}
  {}
  {\bfseries}
  { }
  {  }
  {\thmname{ #1}\thmnumber{ #2}\thmnote{\mdseries \itshape \ #3}}

\theoremstyle{dotless}

\newtheorem*{thm}{Theorem}
\newtheorem*{cor}{Corollary}
\newtheorem*{lem}{Lemma}
\newtheorem*{defn}{Definition}

\begin{document}

\title{Essential Trigonometry Without Geometry}

\author{
John Gresham\\
Bryant Wyatt\\
Jesse Crawford\\
Tarleton State University}
\date{\today}

\begin{titlepage}
\maketitle
\thispagestyle{empty}
\setcounter{page}{0}
\end{titlepage}

\begin{abstract}
The development of the trigonometric functions in introductory texts usually follows geometric constructions using right triangles or the unit circle.  While these methods are satisfactory at the elementary level, advanced mathematics demands a more rigorous approach. Our purpose here is to revisit elementary trigonometry from an entirely analytic perspective.  We will give a comprehensive treatment of the sine and cosine functions and will show how to derive the familiar theorems of trigonometry without reference to geometric definitions or constructions.  
\end{abstract}

\bigskip

\subsection*{Introduction}
As we approach trigonometry from an analytic perspective, our understanding deepens and old theorems become new again.  For this study, we will assume a familiarity with calculus, differential equations, and real analysis. 
  
\subsection*{Definitions and Basic Properties}
We begin by considering the solution of the second-order homogeneous linear
differential equation%
\[
f\,^{\prime \prime }\left( x\right) +f\left( x\right) =0\text{ with }%
f\left( 0\right) =0\text{ and }f\,^{\prime }\left( 0\right) =1.
\]%
By the Existence and Uniqueness Theorem we know that a unique solution
exists [Nagle, Saff, and Snider, p. 171]. If this solution has a power series representation around the ordinary point $x=0$, it must have the form
\[
f\left( x\right) =\sum_{n=0}^{\infty }c_{n}x^{n}
\]%
Note that $f\left( 0\right) =c_{0}=0$ and $f\,^{\prime }\left( 0\right)
=c_{1}=1$. We also have%
\begin{align*}
f\,^{\prime \prime }\left( x\right)  &=\sum_{n=2}^{\infty }\left( n\right)
\left( n-1\right) c_{n}x^{n-2} \\
&=\sum_{n=0}^{\infty }\left( n+2\right) \left( n+1\right) c_{n+2}x^{n}
\end{align*}%
Then%
\[
\sum_{n=0}^{\infty }\left( n+2\right) \left( n+1\right)
c_{n+2}x^{n}+\sum_{n=0}^{\infty }c_{n}x^{n}=\sum_{n=0}^{\infty }\left(
\left( n+2\right) \left( n+1\right) c_{n+2}+c_{n}\right) x^{n}=0
\]%
Since this power series is $0$ for all $x$, we get the general recursion relation%
\[
\left( n+2\right) \left( n+1\right) c_{n+2}+c_{n}=0
\]%
so that 
\[
c_{n+2}=-\frac{c_{n}}{\left( n+2\right) \left( n+1\right) }.
\]%
Because $c_{0}=0$, we have for all even indices $2n$%
\[
c_{2n}=0
\]%
Let us now examine the coefficients with odd indices $2n+1$.
\begin{align*}
c_{1} &=1\text{ \ \ initial condition} \\
c_{3} &=-\frac{1}{3\cdot 2}=-\frac{1}{3!} \\
c_{5} &=-\genfrac{}{}{1pt}{0}{-\frac{1}{3!}}{5\cdot 4}=\frac{1}{5!} \\
c_{7} &=-\genfrac{}{}{1pt}{0}{\frac{1}{5!}}{7\cdot 6}=-\frac{1}{7!}
\end{align*}%
and in general,%
\[
c_{2n+1}=\left( -1\right) ^{n}\frac{1}{\left( 2n+1\right) !}
\]%
The power series about $x=0$ must have the form%
\[
\sum_{n=0}^{\infty }\left( -1\right) ^{n}\dfrac{x^{2n+1}}{\left( 2n+1\right)
!}
\]%
Using the Ratio Test, it is easy to show that this series converges for all
real $x$. \medskip 

The function represented by this power series is the unique solution of the
differential equation\medskip 
\[
f\,^{\prime \prime }\left( x\right) +f\left( x\right) =0\text{ with }%
f\left( 0\right) =0\text{ and }f\,^{\prime }\left( 0\right) =1.
\]%
\medskip We call this function the \textbf{sine }function\textbf{, }denoted $%
\sin x$, or $\sin \left( x\right) $.\textbf{\medskip }

\begin{defn}[Sine Function]
\[
\sin x=\sum_{n=0}^{\infty }\left( -1\right) ^{n}\dfrac{x^{2n+1}}{\left(
2n+1\right) !} 
\]
\end{defn}

We define the \textbf{cosine} to be the derivative of the sine function.

\begin{defn}[\textit{Cosine Function}]
\[
\cos x=\dfrac{d}{dx}\sum_{n=0}^{\infty }\left( -1\right) ^{n}\dfrac{x^{2n+1}%
}{\left( 2n+1\right) !}=\sum_{n=0}^{\infty }\left( -1\right) ^{n}\dfrac{%
x^{2n}}{\left( 2n\right) !}
\]

\end{defn}

The following are elementary consequences of the
definitions.\medskip

\begin{enumerate}
\item $\sin 0=0$

\item $\cos 0=1$

\item The function $\sin x$ is odd because all exponents in its power series
are odd.

\item The function $\cos x$ is even because all exponents in its power
series are even.

\item The functions $\sin x$ and $\cos x$ are both continuous since they are
differentiable.

\item The derivatives of $\sin x$ are cyclic with order four.

\end{enumerate}

\[
\begin{tabular}{|c|c|c|c|c|}
\hline
$f\left( x\right) $ & $f\,^{\prime }\left( x\right) $ & $f\,^{\prime \prime
}\left( x\right) $ & $f\,^{\prime \prime \prime }\left( x\right) $ & $%
f\,^{\prime \prime \prime \prime }\left( x\right) $ \\ \hline
$\sin x$ & $\cos x$ & $-\sin x$ & $-\cos x$ & $\sin x$ \\ \hline
\end{tabular}%
\]

\medskip

\subsection*{Key Theorems}

This section presents the Pythagorean and Sine Sum identities which, along with the smallest positive critical value of $\sin x$, enable the development of several important identities and analytic results in elementary trigonometry.
\bigskip 

First, we prove the Pythagorean Identity.  

\begin{thm}[Pythagorean Identity]

For all $x$,%
\[
\sin ^{2}x+\cos ^{2}x=1 
\]
\end{thm}
\begin{proof}[Proof:\nopunct] Consider the derivative of the left side.

\begin{align*}
\dfrac{d}{dx}\left( \sin ^{2}x+\cos ^{2}x\right) &=2\sin x\cos x+2\cos x\left( -\sin x\right)\\&=0
\end{align*}

\medskip Since the derivative is $0$, $\sin ^{2}x+\cos ^{2}x$ is a constant.

Because $\sin 0=0$, and $\cos 0=1$, this
constant must be $1$. 
\end{proof}

\bigskip

Next, we consider the identity for the sine of the sum of $x$ and $y$. The
proof in most elementary trigonometry texts involves a geometric
construction with triangles or the unit circle. In our geometry-free
approach, we will use only power series.\medskip

\begin{thm}[Sine Sum Identity]For all $x$, $y$, 
\[
\sin \left( x+y\right) =\sin x\cos y+\cos x\sin y 
\]
\end{thm}

\begin{proof}[Proof:\nopunct]

Consider the series expansion%
\[
\sin \left( x+y\right) =\sum_{n=0}^{\infty }\left( -1\right) ^{n}\dfrac{%
\left( x+y\right) ^{2n+1}}{\left( 2n+1\right) !} 
\]%
Now examine the general $n^{th}$ term $a_{n}$ of this series using the Binomial
Theorem:
\begin{align*}
a_{n} 
&=\left( -1\right) ^{n}\dfrac{\left( x+y\right) ^{2n+1}}{\left( 2n+1\right) !}\\
&=\dfrac{\left( -1\right) ^{n}}{\left( 2n+1\right) !}\left( x+y\right)
^{2n+1} \\
&=\dfrac{\left( -1\right) ^{n}}{\left( 2n+1\right) !}\sum_{i=0}^{2n+1}%
\dbinom{2n+1}{i}x^{2n+1-i}y^{i} \\
&=\dfrac{\left( -1\right) ^{n}}{\left( 2n+1\right) !}\sum_{i=0}^{2n+1}%
\dfrac{\left( 2n+1\right) !}{i!\left( 2n+1-i\right) !}x^{2n+1-i}y^{i} \\
&=\left( -1\right) ^{n}\sum_{i=0}^{2n+1}\dfrac{1}{i!\left( 2n+1-i\right) !}%
x^{2n+1-i}y^{i}
\end{align*}%
\medskip This last sum has $2n+2$ terms. We will re-write it as two sums
each having $n+1$ terms.\medskip 

\begin{align*}
\left( -1\right) ^{n}\sum_{i=0}^{2n+1}\dfrac{1}{i!\left( 2n+1-i\right) !}%
x^{2n+1-i}y^{i} &=\left( -1\right) ^{n}\underbrace{\sum_{i=0}^{n}\dfrac{%
x^{2i+1}y^{2n-2i}}{\left( 2i+1\right) !\left( 2n-2i\right) !}}+\left(
-1\right) ^{n}\underbrace{\sum_{i=0}^{n}\dfrac{x^{2n-2i}y^{2i+1}}{\left(
2n-2i\right) !\left( 2i+1\right) !}} \\
&\text{\ \ \ \ \ \ \ \ increasing odd powers of }x\text{ \ \ \ \ \ \ \
decreasing\ even powers of }x \\
&=\dfrac{\left( -1\right) ^{n}}{\left( 2n+1\right) !}\sum_{i=0}^{n}\dfrac{%
\left( 2n+1\right) !x^{2i+1}y^{2n-2i}}{\left( 2i+1\right) !\left(
2n-2i\right) !}+\dfrac{\left( -1\right) ^{n}}{\left( 2n+1\right) !}%
\sum_{i=0}^{n}\dfrac{\left( 2n+1\right) !x^{2n-2i}y^{2i+1}}{\left(
2n-2i\right) !\left( 2i+1\right) !} \\
&=\underbrace{\dfrac{\left( -1\right) ^{n}}{\left( 2n+1\right) !}%
\sum_{i=0}^{n}\dbinom{2n+1}{2i+1}x^{2i+1}y^{2n-2i}}+\underbrace{\dfrac{%
\left( -1\right) ^{n}}{\left( 2n+1\right) !}\sum_{i=0}^{n}\dbinom{2n+1}{2i+1}%
x^{2n-2i}y^{2i+1}} \\
&\text{\ \ \ \ \ \ \ \ \ \ \ \ \ \ \ \ \ \ \ \ \ \ \ \ \ [1] \ \ \ \ \ \ \ \ \ \ \ \ \ \ \ \ \ \ \ \ \ \ \ \ \ \ \ \ \ \ \ \ \ \ \
\ \ \ \ \ \ \ \ \ \ [2]}
\end{align*}%
\medskip This last line represents the $n^{th}$ term of the expansion of $%
\sin \left( x+y\right) $. We now turn our attention to the right side 
\[
\sin x\cos y+\cos x\sin y 
\]%
and consider the series expansion of the term $\sin x\cos y$.\medskip

Since the series for $\sin x$ and for $\cos x$ both converge absolutely, we
can write $\sin x\cos y$ as the Cauchy product of the two series%
\[
\sin x\cos y=\sum_{n=0}^{\infty }c_{n} 
\]%
where 
\[
c_{n}=\sum_{i=0}^{n}a_{i}b_{n-i\,}\text{, \ }n=0,1,2,3,... 
\]%
and the $a_{i}$, $b_{n-i}$ terms come from the series for $\sin x$ and $\cos x$, respectively [Rudin, p. 63ff]. Let us examine the general term $c_{n}$ of
this Cauchy product.%
\begin{align*}
c_{n} &=\sum_{i=0}^{n}\left( -1\right) ^{i}\dfrac{x^{2i+1}}{\left(
2i+1\right) !}\cdot \left( -1\right) ^{n-i}\dfrac{y^{2n-2i}}{\left(
2n-2i\right) !} \\
&=\sum_{i=0}^{n}\left( -1\right) ^{n}\dfrac{x^{2i+1}y^{2n-2i}}{\left(
2i+1\right) !\left( 2n-2i\right) !} \\
&=\dfrac{\left( -1\right) ^{n}}{\left( 2n+1\right) !}\sum_{i=0}^{n}\dfrac{%
\left( 2n+1\right) !}{\left( 2i+1\right) !\left( 2n-2i\right) !}%
x^{2i+1}y^{2n-2i} \\
&=\dfrac{\left( -1\right) ^{n}}{\left( 2n+1\right) !}\sum_{i=0}^{n}\dbinom{%
2n+1}{2n-2i}x^{2i+1}y^{2n-2i}
\end{align*}%
Then the term $c_{n}$ is the odd powers of $x$ in part [1] of the
general binomial expansion above.

\medskip

By switching $x$ with $y$ in the previous equation, we get the general term $d_{n}$ for the Cauchy product of the series for $\sin y$ and $\cos x$.
\begin{align*}
d_{n} &=\dfrac{\left( -1\right) ^{n}}{\left( 2n+1\right) !}%
\sum_{i=0}^{n}\dbinom{2n+1}{2n-2i}y^{2i+1}x^{2n-2i} \\
&=\dfrac{\left( -1\right) ^{n}}{\left( 2n+1\right) !}\sum_{i=0}^{n}\dbinom{%
2n+1}{2n-2i}x^{2n-2i}y^{2i+1} \\
&=\dfrac{\left( -1\right) ^{n}}{\left( 2n+1\right) !}\sum_{i=0}^{n}\dbinom{%
2n+1}{2i+1}x^{2n-2i}y^{2i+1}
\end{align*}%
This matches the even powers of $x$ in part [2] of the general binomial
expansion.


Therefore\medskip
\[a_{n}=c_{n}+d_{n}\]
and 
\[
\sin \left( x+y\right) =\sin x\cos y+\cos x\sin y.
\]
\end{proof}
\medskip

We now turn our attention to a special value, the smallest positive critical value of $\sin x$, a number we will call $Q$.
\bigskip

\begin{thm}[Critical Value]
  There exists a smallest positive critical value of $\sin x$, that is, a smallest positive zero of $\cos x$.
\end{thm}

\begin{proof}

We have already seen that $\cos 0=1$. Now observe that%
\[
\cos 2=1-\frac{2^{2}}{2!}+\frac{2^{4}}{4!}-\frac{2^{6}}{6!}+\cdots 
\]%
We now write 
\begin{align*}
\cos 2 &=\left( 1-\frac{2^{2}}{2!}+\frac{2^{4}}{4!}-\frac{2^{6}}{6!}\right)
+R_{3} \\
&=\left( -\frac{19}{45}\right) +R_{3} \\
&\leq -\frac{19}{45}+\left\vert R_{3}\right\vert
\end{align*}%
The Remainder Theorem for alternating series tells us that%
\begin{align*}
\left\vert R_{3}\right\vert &\leq a_{4}=\frac{2^{8}}{8!}\text{ \ and so} \\
\cos 2 &\leq -\frac{19}{45}+\frac{2}{315}=-\frac{131}{315}
\end{align*}%
\medskip Since $\cos 0>0$ and $\cos 2<0$, by the Intermediate Value Theorem,
there is at least one real number $c\in \left( 0,2\right) $ with $\cos c=0$.
The nonempty set $\left\{ x|\cos x=0\right\} $ is the inverse image of the
closed point set $\left\{ 0\right\}$ under the continuous function $\cos x$.
Therefore the set $\left\{ x|\cos x=0\right\} $ is closed. It follows that the set
\[
\left\{ x|\cos x=0\right\} \cap \left[ 0,2\right] 
\]%
is nonempty, closed, bounded, and is therefore compact [Willard, p. 120]. It
must contain its least element which we shall call, temporarily, $Q$.

\end{proof}

\medskip

\textbf{Definition of $Q$}%
\[
Q=\min \left( \left\{ x|\cos x=0\right\} \cap \left[ 0,2\right] \right) 
\]

\subsection*{Consequences of the Key Theorems}

The Pythagorean Identity leads directly to the following corollary.

\medskip

\begin{cor}For all $x$, 
\[
\left\vert \sin x\right\vert \leq 1\text{ \ and \ }\left\vert \cos
x\right\vert \leq 1. 
\]
\end{cor}

\begin{proof}[Proof:\nopunct]

If $\left\vert \sin x\right\vert > 1$, then $\cos^{2} x < 0$ and $\cos x$ is not a real number. Similarly, if $\left\vert \cos x\right\vert > 1$, then $\sin x$ is not a real number.  In this study, we are restricting our work to real numbers.

\end{proof}

\medskip

The next two corollaries follow from the Pythagorean Identity and the special properties of $Q$.

\pagebreak

\begin{cor} $\sin Q=1$ and $\sin x$ has an absolute maximum value of $1$ at $%
x=Q$.
\end{cor}

\begin{proof}[Proof:\nopunct] Since $\cos 0=1$ and $\cos x$ is an even function, for $x\in ( -Q,Q)$, we have $\cos x>0$. Therefore $\sin x$ is strictly increasing on $(-Q,Q)$. Since $0<Q$ we have $0=\sin 0<\sin Q$.  From the Pythagorean Identity we know that%
\[
\sin ^{2}Q+\cos ^{2}Q=1 
\]%
Since $\cos Q=0$, it must be the case that $\sin Q=1$. We have already
observed that%
\[
\left\vert \sin x\right\vert \leq 1 
\]%
and therefore $1$ is an absolute maximum of $\sin x$.
\end{proof}
\medskip

\begin{cor} The range of $\sin x$ is $\left[ -1,1\right] $.
 
\end{cor}

\begin{proof}[Proof:\nopunct] Because $\sin x$ is an odd function we have $\sin \left(
-Q\right) =-\sin Q=-1$ is an absolute minimum. The range $\left[ -1,1%
\right] $ follows from the continuity of $\sin x$ and the Intermediate Value
Theorem.\end{proof}
Later will will see that the range of cosine is also $[-1,1]$.

\bigskip

Our next two corollaries follow from the Sine Sum Theorem.

\bigskip

\begin{cor} $\sin \left( x-y\right) =\sin x\cos y-\cos x\sin y$
\end{cor}

\begin{proof}[Proof:\nopunct]

Because $\sin x$ is an odd function and $\cos x$ is even, we have the following:

\begin{align*}
\sin \left( x-y\right) &=\sin \left( x+\left( -y\right) \right) \\
&=\sin x\cos \left( -y\right) +\cos x\sin \left( -y\right) \\
&=\sin x\cos y-\cos x\sin y
\end{align*}%
\end{proof}

\medskip

\begin{cor}$
\sin 2x=2\sin x\cos x$
\end{cor}

\begin{proof}[Proof:\nopunct]

\begin{align*}
\sin 2x &=\sin \left( x+x\right) \\
&=\sin x\cos x+\cos x\sin x \\
&=2\sin x\cos x
\end{align*}

\end{proof}

We now consider the cofunction rules that follow from the Sine Sum Identity and the properties of $Q$. We will use these later to show that the sine and cosine functions are periodic.\bigskip

\begin{cor}[Cofunction Rule] $\sin \left( Q-x\right)
=\cos x$
\end{cor}

\begin{proof}[Proof:\nopunct]
\begin{align*}
\sin \left( Q-x\right) &=\sin Q\cos x-\cos Q\sin x \\
&=1\cdot \cos x-0\cdot \sin x \\
&=\cos x
\end{align*}%
\end{proof} 

\begin{cor}[Cofunction Rule]$\cos \left( Q-x\right) =\sin x$
\end{cor}

\begin{proof}[Proof:\nopunct]
\begin{align*}
\cos \left( Q-x\right) &=\sin \left( Q-\left( Q-x\right) \right) \\
&=\sin x
\end{align*}%
\end{proof} 

In the following corollaries we complete the sum, difference, and double angle rules.

\begin{cor} \ $\cos \left( x+y\right) =\cos x\cos y-\sin x\sin y$
\end{cor}

\begin{proof}[Proof:\nopunct] 
\begin{align*}
\cos \left( x+y\right) &=\sin \left( Q-\left( x+y\right) \right) \\
&=\sin \left( \left( Q-x\right) -y\right) \\
&=\sin \left( Q-x\right) \cos y-\cos \left( Q-x\right) \sin y \\
&=\cos x\cos y-\sin x\sin y
\end{align*}%
\end{proof} 
\medskip 

The following corollaries now follow.\bigskip

\begin{cor} \ $\cos \left( x-y\right) =\cos x\cos y+\sin x\sin y$
\end{cor}

\begin{proof}[Proof:\nopunct] 
\begin{align*}
\cos \left( x-y\right) &=\cos(x+(-y)) \\
&=\cos x \cos(-y)- \sin x \sin(-y)  \\
&=\cos x \cos y + \sin x \sin y
\end{align*}%
\end{proof}

\medskip

\begin{cor} \ $\cos 2x =2\cos^{2} x-1$
\end{cor}

\begin{proof}[Proof:\nopunct]

\begin{align*}
\cos 2x &=\cos \left( x+x\right) \\
&=\cos x\cos x-\sin x\sin x \\
&=\cos^{2}x-\sin^{2}x\\
&=\cos^{2}x-(1-\cos^{2}x)\\
&=2\cos^{2} x-1
\end{align*}

\end{proof}

\bigskip

We have seen that the three key theorems have led to the familiar difference formulas as well as double angle formulas.  From these follow the other identities such as half-angle and product-to-sum rules.  In particular, we will later need the
identity%
\[
\cos ^{2}x=\frac{1}{2}+\frac{1}{2}\cos 2x 
\]
\medskip

\subsection*{Periodicity}

We will need the sine and cosine function values of $4Q$ to show periodicity.  Here is a sequence of steps to arrive at this point. 

\medskip

\begin{enumerate}
\item $\sin 2Q=2\sin Q\cos Q=2(1)(0)=0$

\item $\cos 2Q=\sin \left( Q-2Q\right) =\sin \left( -Q\right) =-\sin Q=-1$.\\
From this it follows that the range of $\cos x$ is $[-1,1].$

\item $\sin 3Q=\sin \left( Q+2Q\right) =\sin Q\cos 2Q+\cos Q\sin 2Q=-1$

\item $\cos 3Q=\sin \left( Q-3Q\right) =\sin \left( -2Q\right) =-\sin 2Q=0$

\item $\sin 4Q=2\sin 2Q\cos 2Q=0$

\item $\cos 4Q=\sin \left( Q-4Q\right) =\sin \left( -3Q\right)=-\sin\left(3Q\right)=-(-1) =1\medskip $
\end{enumerate}

We now have the machinery needed to prove the periodicity of $\sin x$ and $%
\cos x$. \medskip

\begin{defn}
A function $f\left( x\right) $ is \textit{periodic}
if there is a positive number $p$ such that%
\[
f\left( x+p\right) =f\left( x\right) 
\]%
for all $x$. If there is a \textsl{smallest} positive number $p$ for which
this holds, then $p$ is called the \textit{period of }$f$.
\end{defn}

\begin{thm}[Periodicity of Sine]  The sine function is periodic and its period is $4Q$.
\end{thm}

\begin{proof}[Proof:\nopunct]  We first show that sine is periodic.
\begin{align*}
\sin \left( x+4Q\right) &=\sin x\cos 4Q+\cos x\sin 4Q \\
&=\sin x\left( 1\right) +\cos x\left( 0\right) \\
&=\sin x
\end{align*}%

This shows that $\sin x$ is periodic, but does not show that the
period is $4Q$. To show that $4Q$ is the period, assume, to the
contrary, that there exists a number $R$ such that $0<4R<4Q$ and for all $x$,%
\[
\sin \left( x+4R\right) =\sin x 
\]%
Observe that $0<R<Q$. For $x\in \left( 0,Q\right) $ we have $\cos x>0$ because
$\cos 0=1$ and $Q$ is the smallest value with $\cos Q=0$. We also have $%
\sin x>0$ since $\sin 0=0$ and $\sin $ is increasing on $\left( 0,Q\right)$. 
Now examine $\sin Q$:
\begin{align*}
\sin Q &=\sin \left( Q+4R\right) \\
&=\sin Q\cos 4R+\cos Q\sin 4R \\
&=\cos 4R \\
&=\cos 2\left( 2R\right) \\
&=2\cos ^{2}\left( 2R\right) -1
\end{align*}

Because $\sin Q=1$,
\begin{align*}
1 &=2\cos ^{2}\left( 2R\right) -1 \\
1 &=\cos ^{2}\left( 2R\right) \\
\cos 2R &=1\text{ or }\cos 2R=-1
\end{align*}%
We now have two cases:\\
\pagebreak

\ Case I: \ $\cos 2R=1.$

\ \ \ \ \ \ \ \ \ \ \ \ \ \ Then by the double angle identity, 
\begin{align*}
2\cos ^{2}R-1 &=1 \\
\cos ^{2}R &=1
\end{align*}

If $\cos^{2} R=1$, then by the Pythagorean Identity, $\sin R=0$, a contradiction
to the fact that $\sin R>0$.

\medskip 

\ Case II: \ $\cos 2R=-1$.

\qquad Then%
\begin{align*}
2\cos ^{2}R-1 &=-1 \\
\cos R &=0
\end{align*}%
This last statement contradicts the choice of $Q$ as the smallest positive
number in $\left[ 0,2\right] $ with $\cos Q=0$. Therefore such a number $R$ does not exist, and the period of $\sin $  is $4Q$.%
\end{proof} 
\bigskip

\begin{cor}[Periodicity of Cosine]The cosine function is periodic with period $4Q$.
\end{cor}

\begin{proof}[Proof:\nopunct]  We can write $\cos x$ as \[\cos x=-\sin(x-Q)\]
Because horizontal translations and vertical rotations about the x-axis do not change the period of a function, $\cos x$ is periodic with period $4Q$.
\end{proof} 
\bigskip

\subsection*{Connection to Geometry}

With this result we now show the connection between the analytic and
geometric approaches to trigonometry. 
\bigskip

\begin{thm}[Connection with $\pi$] \[ \int_{0}^{1}\sqrt{1-x^{2}}\, dx=\frac{Q}{2} \]
\end{thm}

\begin{proof}[Proof:\nopunct] Use the substitution
\[
x=\sin \theta 
\]%
with the values 
\begin{tabular}{c|c}
$x$ & $\theta $ \\ \hline
$0$ & $0$ \\ 
$1$ & $Q$ \\ 
\end{tabular}
\ so that the integral becomes\medskip 
\begin{align*}
\int_{0}^{Q}\sqrt{1-\sin ^{2}\theta }\cos \theta \,d\theta
&=\int_{0}^{Q}\cos ^{2}\theta \,d\theta \\
&=\int_{0}^{Q}\frac{1}{2}\left( 1+\cos 2\theta \right) \,d\theta \\
&=\frac{1}{2}\left[ \theta +\frac{1}{2}\sin 2\theta \right] _{0}^{Q} \\
&=\frac{1}{2}\left[ \left( Q+\frac{1}{2}\sin 2Q\right) -\left( 0+\frac{1}{2}%
\sin \left( 2\cdot 0\right) \right) \right] \\
&=\frac{1}{2}Q
\end{align*}%
\end{proof} 
\medskip 

The integral $\int_{0}^{1}\sqrt{1-x^{2}}\,dx$ represents the
quarter-circle area enclosed by the unit circle, the nonnegative $x$-axis,
and the nonnegative $y$-axis, and so we are led to the conclusion that \[Q=\pi/2\].\medskip 

Using what we have previously developed about multiples of $Q$, we have a table restating the values for sine and cosine in terms of $\pi $ instead of $Q$.\medskip

\begin{center}
\begin{tabular}{c|c|c|r|r|c}
$x$ & $0$ & $\pi /2$ & $\pi $ & $3\pi /2$ & $2\pi $ \\ \hline
$\sin x$ & $0$ & $1$ & $0$ & $-1$ & $0$ \\ 
$\cos x$ & $1$ & $0$ & $-1$ & $0$ & $1$ \\ 
\end{tabular}%
\medskip
\end{center}

From this follows the usual information about the graphs of the sine and
cosine: intervals for positive/negative values, intervals for
increasing/decreasing, local (and absolute) maximums/minimums.\bigskip

Without geometry, we can find the values of sine and cosine of $\dfrac{\pi }{%
4}$, $\dfrac{\pi }{6}$, $\dfrac{\pi }{3}$ using only the sum and difference
identities. We include the development of these values in Appendix A. In Appendix B we present the mathematics that connects the sine and cosine functions, defined here as power series, to the trig functions defined using the unit circle.

 \bigskip
 
\subsection*{Pythagorean Identity Revisited}

We conclude this study with the observation that the \textbf{converse} of the
Pythagorean Identity also holds.\medskip

\begin{thm} If $f:\mathbb{R} \to \mathbb{R}$ is analytic, $f\,^{\prime
}\left( 0\right) =1$, $\ f\left( 0\right) =0$, and $f$ satisfies the
Pythagorean Identity
\[
\left( f\left( x\right) \right) ^{2}+\left( f\,^{\prime }\left( x\right)
\right) ^{2}=1 
\]%
for all $x$, then $f$ $\left( x\right) \equiv \sin x$.
\end{thm}

\begin{proof}[Proof:\nopunct] Differentiation of both sides gives%
\[
2\left( f\left( x\right) \right) f\,^{\prime }\left( x\right) +2f\,^{\prime
}\left( x\right) f\,^{\prime \prime }\left( x\right) =0 
\]%
so that\medskip 
\[
2f\,^{\prime }\left( x\right) \left( f\left( x\right) +f\,^{\prime \prime
}\left( x\right) \right) =0 
\]

Since $f'(0)=1$, and $f$ is analytic, $f'$ is positive on some open interval containing $0$.  Therefore, \textbf{on this interval}, 

\[
f\left( x\right) +f\,^{\prime \prime }\left( x\right) =0,
\]
\medskip 

and $f(x) = \sin(x)$.  Moreover, if two analytic functions agree on an open interval, then they agree on $\mathbb{R}$.

\end{proof} 
\medskip

\subsection*{Summary and Conclusions}

We have developed the theorems and identities of basic trigonometry using the definition of the sine function as the solution, expressed as a power series, of a certain second order linear homogeneous differential equation. The key theorems in this study are the Pythagorean Identity, the Sine Sum Identity, and the special value $Q$, which turned out to be $\pi/2$. From these the other identities follow. The interested reader is referred to Landau, chapter 16, in which the sine and cosine functions are developed from a power series definition. In a brief note, Appendix III in Hardy uses the definition of the inverse tangent function as an integral to lead to the definitions of sine, cosine, and their sum laws.

In a future study we plan to consider a generalization of the sine and cosine functions, and show that versions of the Key Theorems still hold in these settings.

\medskip

\pagebreak

\subsection*{REFERENCES}

\noindent
Hardy, G. H. \ 1921. \textit{A Course of Pure Mathematics, Third Edition}. Mineola, New York:
Dover\\
\indent
Publications, Inc., 2018. An unabridged republication of the work originally published by\\
\indent
the Cambridge University Press, Cambridge.\\

\noindent
Landau, Edmund. 1965. \textit{Differential and Integral Calculus}. Providence, Rhode Island: AMS\\
\indent
Chelsea Publishing. Reprinted by the American Mathematical Society, 2010, from a translation\\
\indent
of Edmund Landau's 1934 \textit{Einf\"{u}hring in die Differentialrechnung und Integralrechnung}.\\

\noindent
Nagle, R. Kent, Edward B. Saff, and Arthur David Snider. 2008. \textit{Fundamentals of
Differential Equations}.\\
\indent
Boston: Pearson Addison Wesley.\\

\noindent
Rudin, Walter. 1964. \textit{Principles of Mathematical Analysis}. New York: McGraw-Hill Book Company.\\

\noindent
Willard, Stephen. 1970. \textit{General Topology}. Reading, Massachusetts: Addison-Wesley Publishing Company.

\pagebreak

\subsection*{Appendix A: Trig Functions of Special Angles}

First we consider $\dfrac{\pi }{4}$.\medskip

\[
\sin \left( \dfrac{\pi }{4}\right) =\cos \left( \dfrac{\pi }{2}-\dfrac{\pi }{%
4}\right) =\cos \left( \dfrac{\pi }{4}\right) 
\]

Since%
\[
1=\sin ^{2}\left( \dfrac{\pi }{4}\right) +\cos ^{2}\left( \dfrac{\pi }{4}%
\right) =2\sin ^{2}\left( \dfrac{\pi }{4}\right) 
\]%
we obtain 
\[
\sin \left( \dfrac{\pi }{4}\right) =\dfrac{\sqrt{2}}{2}=\cos \left( \dfrac{%
\pi }{4}\right) 
\]%
\medskip

To find values for $\sin \dfrac{\pi }{6}$, we need the triple-angle identity
\[
\sin \left( 3\theta \right) =3\sin \theta -4\sin ^{3}\theta 
\]
This follows from expanding $\sin(2\theta+\theta)$. We can now write
\begin{align*}
1 &=\sin \dfrac{\pi }{2} \\
&=\sin \left( 3\cdot \dfrac{\pi }{6}\right) \\
&=3\sin \dfrac{\pi }{6}-4\sin ^{3}\dfrac{\pi }{6}
\end{align*}%
We solve this cubic equation in $\sin \dfrac{\pi }{6}$ to obtain a double
solution $\dfrac{1}{2}$ and single solution $-1$.  Because $\sin \dfrac{\pi 
}{6}>0$,%
\begin{align*}
\sin \dfrac{\pi }{6} &=\dfrac{1}{2}\text{, and} \\
\cos \dfrac{\pi }{6} &=\dfrac{\sqrt{3}}{2}
\end{align*}%
\medskip Here's a summary table:%
\[
\begin{tabular}{c|c|c|c|c|c}
$x$ & $0$ & $\pi /6$ & $\pi /4$ & $\pi /3$ & $\pi /2$ \\ \hline
$\sin x$ & $0$ & $1/2$ & $\sqrt{2}/2$ & $\sqrt{3}/2$ & $1$ \\ 
$\cos x$ & $1$ & $\sqrt{3}/2$ & $\sqrt{2}/2$ & $1/2$ & $0$ \\ 
\end{tabular}%
\]

\subsection*{Appendix B: Connection to Unit Circle Trigonometry}
\medskip 

\begin{lem}

If $f$ is continuous on $\left[ a,b\right] $ and strictly
increasing on $\left( a,b\right) $, then $f$ is strictly increasing on $%
\left[ a,b\right] $.

\end{lem}

\begin{proof}[Proof:\nopunct] We first show that for any $x$ in the interval $\left(
a,b\right) $, we must have $f\left( a\right) <f\left( x\right) $. Assume, to the
contrary, that there exists $c$, $\ a<c<b$, such that $f\left( a\right) \geq
f\left( c\right) $.\medskip 

Case I: \ $f\left( a\right) >f\left( c\right) $. \ \ Let 
\[
\varepsilon =\frac{f\left( c\right) -f\left( a\right) }{2} 
\]%
Then by right-hand continuity of $f$ at $a$, there exists $\delta $, $%
0<\delta <c-a$, such that if $a<x<a+\delta $ then%
\[
f\left( a\right) -\varepsilon <f\left( x\right) <f\left( a\right)
+\varepsilon 
\]%
Then%
\begin{align*}
f\left( x\right) &>f\left( a\right) -\frac{f\left( c\right) -f\left(
a\right) }{2} \\
&=\frac{f\left( a\right) +f\left( c\right) }{2} \\
&>\frac{f\left( c\right) +f\left( c\right) }{2} \\
&=f\left( c\right)
\end{align*}%
Since $x$, $c$ are both in $\left( a,b\right) $ \ and $x<a+\delta <a+\left(
c-a\right) =c$, we must have $f\left( x\right) <f\left( c\right) $, a
contradiction. Therefore it cannot be the case that $f\left( a\right) \geq
f\left( c\right) $ for some $c\in \left( a,b\right) $.\medskip

Case II: \ $f\left( a\right) =f\left( c\right) $.  Then consider $\frac{a+c}{2}\in \left( a,c\right) \subset \left( a,b\right) $.  Since $f$ is strictly
increasing on $\left( a,b\right) $, it follows that 
\[
f\left( \frac{a+c}{2}\right) <f\left( c\right) =f\left( a\right) 
\]
and we have the situation of Case I.\medskip

A similar argument shows that if $a<x<b$, then $f\left( x\right) <f\left(
b\right) $.

\end{proof}

\bigskip

\begin{thm} 
The function $\sin x$ is strictly increasing on $\left[
-\frac{\pi }{2},\frac{\pi }{2}\right] $.
\end{thm}

\begin{proof}[Proof:\nopunct]
By our development and definition of $\frac{\pi }{2}$
as least in $\left[ 0,2\right] $ with $\cos \left( \frac{\pi }{2}\right) =0$%
, $\cos x$ is positive on $\left[ 0,\frac{\pi }{2}\right) .$ Since it is an
even function, it is positive on $\left( -\frac{\pi }{2},\frac{\pi }{2}%
\right) $ and therefore $\sin x$ is increasing on $\left( -\frac{\pi }{2},%
\frac{\pi }{2}\right) $. The function $\sin x$ is differentiable and
therefore continuous at all $x$, and by the lemma must be increasing on $%
\left[ -\frac{\pi }{2},\frac{\pi }{2}\right] $.

\end{proof}

\medskip

We see, then, that the sine function restricted to $\left[ -\frac{\pi }{2},%
\frac{\pi }{2}\right] $ is a bijection onto $\left[ -1,1\right] $.\medskip

\begin{defn}[Inverse Sine Function]\ The \textit{inverse sine function of $x$}, written here as $\arcsin x$, is defined: 
\[
\arcsin x:\left[ -1,1\right] \rightarrow \left[ -\frac{\pi }{2},\frac{\pi }{2%
}\right] 
\]
is the inverse of the sine function restricted to the domain $\left[ -\frac{\pi }{2},%
\frac{\pi }{2}\right] $.\bigskip

\end{defn}

We now consider the derivative of $\arcsin x$, $-1<x<1$.\bigskip

\begin{thm}
If $-1<x<1$ then $\frac{d}{dx}\arcsin x=\frac{1}{\sqrt{%
1-x^{2}}}$.
\end{thm}

\begin{proof}[Proof:\nopunct]
If $y=\arcsin x$, $-1<x<1$, then $-\frac{\pi }{2}<y<%
\frac{\pi }{2}$ and 
\[
\sin y=x 
\]
Using implicit differentiation, 
\[
\cos y\frac{dy}{dx}=1 
\]%
and 
\[
\frac{dy}{dx}=\frac{1}{\cos y} 
\]%
What is $\cos y=\cos (\arcsin x)$? By the Pythagorean Identity, 
\begin{align*}
\cos ^{2}y+\sin ^{2}y &=1 \\
\cos ^{2}y+x^{2} &=1
\end{align*}%
and so $\cos y=\sqrt{1-x^{2}}$ or $\cos y=-\sqrt{1-x^{2}}$.  Since $-\frac{%
\pi }{2}<y<\frac{\pi }{2}$, we have $\cos y>0$.  Thus $\cos y=\sqrt{1-x^{2}}$ and%
\[
\frac{dy}{dx}=\frac{1}{\sqrt{1-x^{2}}} 
\]
\end{proof}

\bigskip

\begin{cor}
 If $-1<x<1$, then%
\[
\arcsin x=\int_{0}^{x}\frac{1}{\sqrt{1-t^{2}}}dt 
\]
\end{cor}

\begin{proof}[Proof:\nopunct]

This follows from the Fundamental Theorem of Calculus.

\end{proof}

\bigskip 

Does this equation hold for $x=1$ and for $x=-1$?\medskip

\begin{thm}

\[
\int_{0}^{1}\frac{1}{\sqrt{1-t^{2}}}dt=\frac{\pi }{2} 
\]
\end{thm}

\begin{proof}[Proof:\nopunct] Let $g\left( x\right) =\displaystyle\int_{0}^{x}\dfrac{dt}{\sqrt{%
1-t^{2}}}$, $0\leq x\leq 1$.

\medskip

What is the value of $g\left( 1\right) $, an improper integral? First, we
know that $g\left( x\right) =\arcsin x$ for $x\in \lbrack 0,1)$ and that%
\[
\sin \dfrac{\pi }{4}=\cos \left( \dfrac{\pi }{2}-\dfrac{\pi }{4}\right)
=\cos \left( \dfrac{\pi }{4}\right) 
\]%
Since $\sin ^{2}x+\cos ^{2}x=1$, we have 
\[
\sin \dfrac{\pi }{4}=\cos \dfrac{\pi }{4}=\frac{\sqrt{2}}{2} 
\]%
Therefore%
\[
g\left( \tfrac{\sqrt{2}}{2}\right) =\dfrac{\pi }{4} 
\]%
We are now ready to find $g\left( 1\right) $.%
\begin{align*}
g\left( 1\right) &=\lim_{b\rightarrow 1^{-}}\int_{0}^{b}\frac{dt}{\sqrt{%
1-t^{2}}} \\
&=\int_{0}^{\sqrt{2}/2}\frac{dt}{\sqrt{1-t^{2}}}+\lim_{b\rightarrow
1^{-}}\int_{\sqrt{2}/2}^{b}\frac{dt}{\sqrt{1-t^{2}}} \\
&=g\left( \tfrac{\sqrt{2}}{2}\right) +\lim_{b\rightarrow 1^{-}}\int_{\sqrt{%
2}/2}^{b}\frac{dt}{\sqrt{1-t^{2}}} \\
&=\dfrac{\pi }{4}+\lim_{b\rightarrow 1^{-}}\int_{\sqrt{2}/2}^{b}\frac{dt}{%
\sqrt{1-t^{2}}}
\end{align*}%
With the substitution $u=\sqrt{1-t^{2}}$ we obtain for this last integral%
\begin{align*}
\lim_{b\rightarrow 1^{-}}\int_{\sqrt{2}/2}^{b}\frac{dt}{\sqrt{1-t^{2}}}
&=\lim_{b\rightarrow 1^{-}}\int_{\sqrt{2}/2}^{\sqrt{1-b^{2}}}\dfrac{1}{u}%
\cdot \dfrac{-u}{\sqrt{1-u^{2}}}du \\
&=-\lim_{b\rightarrow 1^{-}}\int_{\sqrt{2}/2}^{\sqrt{1-b^{2}}}\dfrac{1}{%
\sqrt{1-u^{2}}}du \\
&=-\int_{\sqrt{2}/2}^{0}\dfrac{1}{\sqrt{1-u^{2}}}du \\
&=\int_{0}^{\sqrt{2}/2}\frac{du}{\sqrt{1-u^{2}}} \\
&=g\left( \tfrac{\sqrt{2}}{2}\right) \\
&=\frac{\pi }{4}
\end{align*}%
so that%
\[
g\left( 1\right) =g\left( \tfrac{\sqrt{2}}{2}\right) +g\left( \tfrac{\sqrt{2}%
}{2}\right) =\dfrac{\pi }{2} 
\]%

\end{proof}

\bigskip

\begin{cor} $\displaystyle\int_{0}^{-1}\frac{1}{\sqrt{1-t^{2}}}dt=-\frac{\pi }{2}$
\end{cor}

\begin{proof}[Proof:\nopunct]
Since $\frac{1}{\sqrt{1-t^{2}}}$ is an even function,
for $0\leq a<1$, we have%
\[
\int_{-a}^{0}\frac{1}{\sqrt{1-t^{2}}}dt=\int_{0}^{a}\frac{1}{\sqrt{1-t^{2}}}%
dt 
\]

and 
\begin{align*}
\int_{0}^{-1}\frac{1}{\sqrt{1-t^{2}}}dt &=-\int_{-1}^{0}\frac{1}{\sqrt{%
1-t^{2}}}dt \\
&=-\lim_{a\rightarrow 1^{-}}\int_{-a}^{0}\frac{1}{\sqrt{1-t^{2}}}dt \\
&=-\lim_{a\rightarrow 1^{-}}\int_{0}^{a}\frac{1}{\sqrt{1-t^{2}}}dt \\
&=-\frac{\pi }{2}
\end{align*}%
\end{proof}

\bigskip 

We now complete the connection to unit circle trigonometry. 
\medskip

\begin{thm}If $-1\leq a\leq b\leq 1$ then the arc length of the graph
of $y=\sqrt{1-x^{2}}$ from $x=a$ to $x=b$ is \[\arcsin b-\arcsin a.\]
\end{thm}
\begin{proof}[Proof:\nopunct] First, note that%
\[
\frac{d}{dx}\sqrt{1-x^{2}}=-\frac{x}{\sqrt{1-x^{2}}} 
\]%
Using the arc length formula and the previous result, 
\begin{align*}
\int\nolimits_{a}^{b}\sqrt{1+\left( \frac{dy}{dx}\right) ^{2}}
&=\int_{a}^{b}\sqrt{1+\left( -\frac{x}{\sqrt{1-x^{2}}}\right) ^{2}}\enspace dx \\
&=\int_{a}^{b}\sqrt{1+\frac{x^{2}}{1-x^{2}}}\enspace dx \\
&=\int_{a}^{b}\sqrt{\frac{1}{1-x^{2}}}\enspace dx \\
&=\int_{a}^{b}\frac{1}{\sqrt{1-x^{2}}}\enspace dx \\
&=\arcsin b-\arcsin a
\end{align*}

\end{proof} 
\bigskip

In the particular case that $b=1$, we have that the arc length $s$ along
the upper unit circle from $x=a$ to $x=1$ is%
\[
s=\frac{\pi }{2}-\arcsin a
\]%
Then 
\begin{align*}
\cos s &=\cos \left( \frac{\pi }{2}-\arcsin a\right)  \\
&=\sin \left( \arcsin a\right)  \\
&=a
\end{align*}%
and%
\begin{align*}
\sin s &=\sin \left( \frac{\pi }{2}-\arcsin a\right)  \\
&=\cos \left( \arcsin a\right)  \\
&=\sqrt{1-a^{2}}
\end{align*}%
This shows that a point $P\left( a,\sqrt{1-a^{2}}\right) $ on the upper unit
circle with $-1\leq a\leq 1$ has coordinates $P\left( \cos s,\sin s\right)$ 
where $s$ is the arc length along the upper unit circle from the point $P$ to $%
A\left( 1,0\right) $. This arc length is the definition of the radian
measure of angle $AOP$ where $O=\left( 0,0\right) $.\medskip 

The connection from geometry-free trigonometry to unit circle trigonometry is complete.

\end{document}